\documentclass[12pt]{article}
\usepackage[latin1]{inputenc}
\usepackage{fullpage}
\usepackage{amsfonts}
\usepackage{latexsym,amssymb}
\usepackage{amsmath}
\usepackage{euscript}
\usepackage{natbib}
\usepackage{graphics}
\usepackage{color}
\usepackage{multirow}

\title{Adaptive Strategies for The Open-Pit Mine\\ Optimal Scheduling Problem} 
\author{Michel \textsc{De Lara}\footnote{%
Universit\'e Paris--Est, \textsc{CERMICS}, 
 6--8 avenue Blaise Pascal, 77455 Marne la Vall\'ee Cedex 2, France. 
Corresponding author:  delara@cermics.enpc.fr, fax +33164153586}, 
Nelson \textsc{Morales}\footnote{\textsc{Delphos} Mine Planning Laboratory,
Advanced Mining Technology Center, University of Chile}, 
Nathana\"el \textsc{Beeker}\footnote{%
Universit\'e Paris--Est, \textsc{CERMICS}, 
 6--8 avenue Blaise Pascal, 77455 Marne la Vall\'ee Cedex 2, France.} }

\newtheorem{proposition}{Proposition}

\newenvironment{proof}{\small{\bf Proof.}}{\hfill$\Box$\normalsize
\bigskip}

\def\mathscr{\EuScript}
\newcommand{\tribu}[1]{\mathscr{#1}}                        
\newcommand{\borel}[1]{\tribu{B}_{#1}^{\mathrm{o}}}         

\newcommand{\BLOCKSET}{\mathcal{B}}
\newcommand{\IR}{\mathbb{R}}
\newcommand{\IN}{\mathbb{N}}
\newcommand{\ARCSET}{\mathcal{A}}
\newcommand{\RESSET}{\mathcal{R}}
\newcommand{\DISCOUNTRATE}{\rho}

\newcommand{\INDEX}{{\tt idx}}
\newcommand{\byvar}[1]{\mathtt{y}_{#1}}
\newcommand{\atvar}[1]{\tiny{\Delta} \byvar{#1}}
\newcommand{\TSET}{\mathbf{T}}
\newcommand{\OPBSP}{{\tt OPBSP}}

\newcommand{\NPV}{\mathtt{NPV}}

\def\RR{{\mathbb R}} 
\def\EE{{\mathbb E}}

\def\PP{{\mathbb P}}

\def\UU{{\mathbb U}}
\def\BB{{\mathbb B}}
\def\XX{{\mathbb X}}
\def\WW{{\mathbb W}}

\def\CC{{\mathbb C}}

\def\defegal{:=} 
\def\text#1{\quad\mbox{#1}\quad} 
\def\mtext#1{\,\mbox{#1}\,} 

\def\Gittins{^{G}}
\def\opt{^{\star}}
\def\upper{_{\sharp}}

\def\1{{\mathbf 1}}

\def\horizon{T}
\def\state{x}
\def\State{X}
\def\control{c}
\def\CONTROL{\CC}
\def\uncertainty{\omega}
\def\Uncertain{\WW}
\def\UNCERTAIN{\Omega}

\def\DYNAMICS{F}
\def\dec{\BB} 

\def\sdo{\mathbb A} 

\def\capacity{v}
\def\discount{\rho}

\def\horizontal{c}
\def\HORIZONTAL{C}
\def\setcolumns{\CC}
\def\depth{d}
\def\DEPTH{D}
\def\surface{\horizontal}
\def\neighbors{\mathcal{M}}

\def\numb{N} 
\def\block{b}
\def\setblocks{\mathcal{B}} 
\def\gain{\mathtt{J}}
\def\setpolicies{\mathcal{P}^{ad}} 
\def\policy{\mathtt{P}}
\def\price{Price}
\def\amount{Ore}
\def\worth{w}
\def\info{\mathcal{I}}
\def\deci{u}
\def\setdecis{\UU}
\def\cost{Cost}

\def\nba{l} 

\def\norme#1{\|#1\|}
\def\rent{\mathtt{V}} 
\def\VALUE{\mathbb{V}}
\def\Value{\mathtt{J}}

\def\retirement{\infty}

\def\HISTORY{{\mathbb H}}

\begin{document}
\maketitle

\begin{abstract}
 Within the mining discipline, mine planning is the component that
studies how to transform the
information about the ore resources into value for the owner.
For open-pit mines, an optimal block scheduling maximizes the
discounted value of the extracted blocks (period by period), called the
\emph{net present value} (NPV). However, 
to be feasible, a mine schedule must respect the slope constraints. 
The optimal {\em open-pit block scheduling problem}
(\OPBSP) consists, therefore, in finding such an optimal schedule.
On the one hand, we introduce the dynamical optimization approach to 
mine scheduling in the deterministic case, and we propose a
class of (suboptimal) adaptive strategies, 
the so-called \emph{index strategies}. We show
that they provide upper and lower bounds for the $\NPV$,
and we provide numerical results.
On the other hand, we introduce a theoretical framework for
\OPBSP\ under uncertainty and learning. 
\end{abstract}


\section{Introduction} 
\label{sec:intro}

Within the mining discipline, mine planning is the component that
studies how to transform the
information about the ore resources into value for the owner. Among
the first decisions taken in the mine planning process is the choice of
an exploitation method: it can be open-pit, that is achieved by
digging from the surface, or it can be underground mining, that is done by
constructing shafts and tunnels to access the mineralized zones.
Other relevant products of the planning process are the {\em
production plan}, that indicates how much will be produced at each
time period, and the {\em mine scheduling}, that backs up the
production plan by specifying what parts of the mine will be extracted
in order to reach the production.
A mine scheduling is constructed by means of a {\em block model},
which is a partition of the terrain into a 3-dimensional array of
regular blocks.  For each block, geostatisticians construct
estimations on the different parameters like ore content, density, etc.
The block model is considered an input to the mine planning
process. 

The operation of a mine is constrained by the overall capacity of
transportation, which is translated into a number of tons per period
(for example, a number of tons per day) and therefore in the number of
blocks that can be extracted from the mine. Similarly, the overall
tonnage of blocks for processing is also bounded by the plant
processing capacity. Notice that, in the case of open-pit mines, not
all blocks qualify for processing as an important part of the blocks
may not contain enough material to have revenue but must be extracted
in order to access attractive blocks.

Open-pit mines are also ``special'' in the sense that extraction
must respect slope constraints: in order to reach blocks
by digging from the surface, there is a minimum set of blocks that
have to be extracted before. Indeed, the shape of the pit must be such that the
stability of the walls and the accessibility are possible. This
translates into a set of {\em precedence} constraints between the
blocks.
Other additional constraints to the operation of the mine may include
{\em blending} constraints, which limit the average value of processed
blocks for a certain attribute (like rock hardness or pollutant
contents).

Considering all these elements, a {\em mine scheduling} can be seen as a
(non injective) mapping from the set of blocks towards the time
periods. Several blocks share the same extraction time. 
An optimal block scheduling maximizes the
discounted value of the extracted blocks (period by period), called the
\emph{net present value} ($\NPV$). However, 
to be feasible, a mine schedule must respect the capacity,
blending and slope constraints. The optimal
{\em open-pit block scheduling problem}
(\OPBSP) consists, therefore, in finding such an optimal block
scheduling.

Related to block scheduling, and central in this article,
is the notion of {\em block sequence}. A block sequence is a total
order on the set of blocks, such that a larger rank means a later
extraction (due to precedence constraints). Block sequences can be easily
converted into block schedules by grouping blocks so that the overall
capacities and blending constraints are satisfied (or, equivalently,
replacing the slope constraints by the precedences given by the
sequence).

The \OPBSP\ is mostly formulated in a {\em deterministic} setting, where 
all values are
supposed to be known to the planner \emph{before} the planning phase:
block model, prices and the operation of the mine (no failure).
The traditional approach to optimal \OPBSP\ uses Binary Integer
Programming (see Appendix~\ref{sec:combinatorial}).
A very general formulation of \OPBSP{} is due to Johnson
\citep*{Joh68,Joh69}, which presented the problem of block scheduling
under slope, capacity and blending constraints (the last ones given by
ranges of the processed ore grade) within a multi-destination setting
(that is, the optimization procedure yields as an output the
process to apply 
to a given block). Unfortunately, the computational capabilities at
the time made impossible to solve the formulation of Johnson for
realistic case studies.
Alternatively to the work of Johnson, Lerchs and Grossman
\citep*{LeGr65} proposed a very simplified version of \OPBSP{} in
which block destinations are fixed in advance and the only constraint
considered is the slope constraint, that is, the problem reduces to
select a subset of blocks such that the contained value is maximized
while the precedence constraint induced by the slope angles are held.
This problem is known as the {\em ultimate pit} or {\em final pit}
problem.  Lerchs and Grossman also presented two key results: 
i) an efficient algorithm for solving the ultimate pit problem;
ii) reducing the economic value of any given block makes the
  optimal solution of the ultimate pit problem to {\em shrink} (that
  is, if the values of the blocks decrease, the new solution is a
  subset of the original one).
These two properties allow to produce nested pits and therefore, by
trial and error, to introduce time and look for block sequences that
satisfy other constraints like capacity. 
More detailed reviews can be found in \citep*{NRW+10} (for a
broad survey on operations research in mining) and 
\citep*{CEG+11} (for the specific case of open-pit).
Finally, an approach closer to the one that will be taken in this
article is due to 
\citep*{Goodwin-Seron-Middleton-Zhang-Hennessy-Stone-Menabde:2006}
which abstract the mine as a set of columns and embed the problem in
the context of control theory.

Regarding mine planning under {\em uncertainty}, since the beginning of the
nineties, an increasing number of open-pit mining strategies with
uncertainty have been developed, following two articles by Ravenscroft
\citep*{Raven92}, and Denby and Schofield \citep*{DS95}. The first one
presents the \emph{conditional simulation}, which is a technique used,
for a mine with a known distribution, to generate sets of equally
probable profiles called \emph{scenarios}. We shall not dwell on the
issue of the design of statistical models  of ore distribution with
uncertainty, using geostatistical tools such as \emph{kriging}  or
others \citep*{Krige84,Jour83,Dowd89,Laj90,Sichel95}, and their
simulation.  The Denby and Schofield \citep*{DS95} paper explains how
to include uncertainty in a genetic algorithm, without precisely
fixing the probabilistic frame. 
Since almost two
decades, most of the stochastic models are based on the Ravenscroft
approach, and present heuristics using a predefined set of scenarios.
Dimitrakopoulos has been one of the driving force behind this trend,
and has developped a large number of scenario-based strategies
\citep*{GD04,DR04,DMR07,DR08}. The solution is generally searched as a
planning, that is an open-loop strategy: we have to plan and apply the
entire scheduling without modifying it in the process of extraction,
even if we get more information on the profile by discovering the
exact value of the blocks.  
Golamnejad, Osanloo, and Karimi \citep*{GOK06} and Boland, Dumitrescu
and Froyland \citep*{BDF08} have also developped scenario-based
strategies with a well defined mathematical and probabilistic
framework: stochastic programming on a scenario tree. This allows
solutions to be defined on a tree rather than only on a line (time), which
clearly is an improvement.
We are interested in how the mine scheduling optimization problem is
formulated and possibly solved under uncertainty.  
We aim at designing solutions as \emph{adaptive} strategies.

The paper is organized as follows, where our objectives are twofold.
On the one hand, we introduce in Section~\ref{sec:dynamical-approach}
the dynamical optimization approach to 
mine scheduling in the deterministic case. 
In Section~\ref{sec:indices}, we propose a
class of (suboptimal) adaptive strategies  
to attack the optimal \OPBSP, the so-called \emph{index strategies}. We show
that they provide upper and lower bounds for the $\NPV$.
We provide numerical results  in Section~\ref{sec:results}.
On the other hand, we introduce in
Section~\ref{sec:uncertainty-framework}  a theoretical framework for
\OPBSP\ under uncertainty and learning.

\section{The dynamical approach to open-pit block scheduling}
\label{sec:dynamical-approach}

As in \citep*{Goodwin-Seron-Middleton-Zhang-Hennessy-Stone-Menabde:2006},
we define the mine state as a collection of pit depths at a certain
number of surface locations and we represent the evolution
of this state via a dynamic model that uses mining action as
control input.
In this setting, an admissible profile is one that respects
local angular constraints at each point, and the
open-pit mine optimal scheduling problem consists in finding a 
sequence of blocks and admissible profiles which maximizes 
the intertemporal discounted extraction profit. 

\subsection{A state control dynamical model}
\label{ss:state}

To simplify the description of the algorithms in this section, we will
identify the blocks by vertical position
$\depth\in\{1,\ldots,\DEPTH\}$ ($\depth$ for depth) and by its
horizontal position $\horizontal\in\setcolumns$ ($\horizontal$ for
column).  In the sequel, it will also be convenient to see the mine as
a collection of columns $\setcolumns$ of cardinal $\HORIZONTAL$
indexed by $\horizontal$, each column containing $\DEPTH$ blocks.
We assume that blocks are extracted sequentially under the following
hypothesis:
\begin{itemize}
\item it takes one time unit to extract one block (thus, the time unit is different from the one in Appendix~\ref{sec:combinatorial});
\item only blocks at the surface may be extracted;
\item a block cannot be extracted if the slope made with its
 neighbors is too high, 
due to geotechnical constraints on mine wall slopes;
\item a retirement option is available where no block is extracted.
\end{itemize}

Denote discrete time by $t=t_{0},\ldots,\horizon$, where the horizon 
$\horizon$ may be finite or infinite.
At time $t$, the \emph{state} of the mine is a \emph{profile} 
\begin{equation}
\state(t)=\big(\state_{\horizontal}(t)\big)_{\horizontal\in\setcolumns} 
\in \XX = \{1,\ldots,\DEPTH+1\}^\HORIZONTAL
\end{equation}
where $\state_{\horizontal}(t)\in\{1,\ldots,\DEPTH+1\}$ 
is the vertical position of the top block with 
horizontal position $\horizontal\in\setcolumns$.


An admissible profile is one that respects
local angular constraints at each point, due to physical requirements. 
A state $\state=(\state_{\horizontal})_{\horizontal\in\setcolumns}$ is said to be 
\emph{admissible} if the geotechnical slope constraints
are respected in the sense that
\begin{equation}
\norme{ \state_{\horizontal'}-\state_{\horizontal} }  \leq 1 \; , \quad 
\forall \horizontal'\in\neighbors(\horizontal) \; , \quad 
\horizontal\in\setcolumns \; ,
\label{eq:slope_constraints}
\end{equation}
where $\neighbors(\horizontal)$ is the set made of columns adjacent to column  $\horizontal$. 
Denote by $\sdo \subset \XX$, 
the set of admissible states satisfying the above slope
constraints~\eqref{eq:slope_constraints}. Notice that $\norme{
  \state_{\horizontal'}-\state_{\horizontal} }  \leq 1$ may be replaced
by $\norme{ \state_{\horizontal'}-\state_{\horizontal} }  \leq k$
according to slope constraints, or even by non-isotropic local slope
constraints. Implicitely, all cuboids have the same dimensions, but we
could deal with less regular situations. 


A decision is the selection of a column in 
$\CONTROL$,
the top block of which will be extracted.
A decision may also be the retirement option, that we
shall identify with an additional fictituous column denoted $\retirement$. 
Thus, a decision $\control$ is an element of the set 
\begin{equation}
\overline{\CONTROL} = \CONTROL \cup \{\retirement\} \; .
\end{equation}
The relation between columns sequencing and blocks scheduling is explicited in \S\ref{sec:columns_blocks} in the Appendix. 

At time $t$, if a column $\control(t) \in \{1,\ldots,\HORIZONTAL\} $ 
is chosen at the surface of the open-pit mine, 
the corresponding block is extracted and the profile 
$\state(t)=\big(\state_{\depth}(t)\big)_{\depth\in\setcolumns}$ becomes 
\begin{equation*}
\state_\depth(t+1)=
\left\{
\begin{array}{ll}
\state_\depth(t)+1 & \text{if} \depth=\control(t)\\
\state_\depth(t) & \text{else.}
\end{array}
\right.
\end{equation*}
In case of retirement option $\control(t)=\retirement$, 
then $\state(t+1)=\state(t)$ and the profile does not change.
In other words, the dynamics is given by 
$\state(t+1)=\DYNAMICS\big(\state(t),\control(t)\big)$ where 
\begin{equation}
\DYNAMICS_\depth(\state,\control) = \left\{ \begin{array}{lcl}
 \state_\depth +1 & \mtext{if} & \depth = 
\control \in\setcolumns \\
 \state_\depth & \mtext{if} & \depth \not = \control  \mtext{ or } 
\depth=\retirement\, .
\end{array} \right.
\label{eq:dynamics}
\end{equation}
Indeed, the top block of column $\depth$ is no longer at depth
$\state_\depth(t)$ but at $\state_\depth(t)+1$, while all other top blocks remain.
Of course, not all decisions $\control(t)=\depth$ are possible 
either because there are no blocks left in column 
$\depth$ ($\state_\depth=\DEPTH+1$) or because of slope constraints.

When in state $\state \in \sdo$, the decision $\control \in \CONTROL$ 
is \emph{admissible} if the future profile
$\DYNAMICS(\state,\control) \in \sdo$, namely if it satisfies the
geotechnical slope constraints.
This may easily be transformed into a condition $\control \in \dec(\state)$, 
where 
\begin{equation}
\dec(\state) \defegal 
\{ \control \in \CONTROL \mid \DYNAMICS(\state,\control) \in \sdo \} \, .
\label{eq:control_constraints}
\end{equation}

\subsection{Intertemporal profit maximization}

The open-pit mine optimal scheduling problem consists of finding a 
sequence of admissible blocks which maximizes an intertemporal discounted
extraction profit.  
It is assumed that the value of blocks differs in depth and column
because richness of the mine is not uniform among the zones as well as
costs of extraction. 
The profit model states that each block has an economic value 
$\rent(\depth,\surface) \in \RR$, supposed to be known (in the deterministic case).
By convention $\rent(\depth,\retirement)=0$ 
when the retirement option is selected.
Selecting a column $\control(t) \in \CONTROL$ at the surface of the 
open-pit mine, and extracting the corresponding block
at depth $\state_{\control(t)}(t)$ yields the
value $\rent\big(\state_{\control(t)}(t),\control(t)\big)$. 
When $\control(t)=\retirement$, there is no corresponding block and the following
notation $\state_{\control(t)}(t)=\state_{\retirement}(t)$ is meaningless, but this is without 
incidence since the value 
$\rent\big(\state_{\retirement}(t),\retirement\big)=0$.

With a discounting factor function
$\discount(t) $ (for instance, $\discount(t)=\discount^t$, or $\discount(t)=\discount^{y(t)}$ for a yearly discount, where 
$y(t)=\lfloor \frac{t}{\capacity} \rfloor$ is the ``year" of time $t$ 
and $\capacity$ is the number of blocks extracted per year), 
the value of a sequence (finite or infinite)
\begin{equation}
 \control(\cdot) \defegal \big( \control(t_{0}), \ldots, \control(\horizon) \big) 
\end{equation}
is given by the criterion
\begin{equation}
 \Value\big( \control(\cdot) \big) \defegal
\sum_{t=t_{0}}^{\horizon} \discount(t) 
\rent\big(\state_{\control(t)}(t),\control(t)\big) \; . 
\label{eq:criterion}
\end{equation}
Finding the value of the mine is solving the optimization problem 
\begin{equation}
\Value\opt \defegal 
\max \{ \sum_{t=t_{0}}^{\horizon} \discount(t) \rent\big(\state_{\control(t)}(t),\control(t)\big) \, , \quad 
\big( \control(\cdot), \state(\cdot) \big) \, , \, 
\control(t) \in \dec\big(\state(t)\big) \} \; ,
\label{eq:intertemporal_profit_maximization}
\end{equation}
where the maximum is over among all sequences \( \big( \control(\cdot), \state(\cdot) \big) \) which satisfy the slope constraints~\eqref{eq:control_constraints}. 
Any such sequence  \( \big( \control\opt(\cdot), \state\opt(\cdot) \big) \) 
such that 
\(  \Value\big( \control\opt(\cdot) \big) = \Value\opt \)
is an \emph{optimal scheduling sequence}. 


\subsection{Dynamic programming equation and the curse of dimensionality}

Theoretically, the open-pit mine optimal scheduling problem can be solved
by dynamic programming \citep*{Bellman:1957,Whittle:1982,Bertsekas:2000}.
It is well known that the dynamic programming approach suffers from 
the curse of dimensionality.
Indeed, to give a flavor of the numerical complexity of the problem,
the set $\sdo$ of acceptable states has a cardinal of order 
$2^{10} \times 3^4 = 82~944$ for a cubic $4 \times 4 \times 4$ mine, 
and of order $2^{16} \times 3^8 \times 4 \approx 1.72 \times 10^9$
for a cubic mine with 5 lateral blocks ($5 \times 5 \times 5$ cuboids).

Nevertheless, usual mines can reach more than $10^6$ blocks, and the dynamic programming approach will not be usable in practice, without further state reduction.

\section{Index strategies} 
\label{sec:indices}

The dynamic programming equation 
\( \VALUE(t,\state) = \max_{\control \in \dec(\state) } 
\Big(\discount(t) \rent(\control,\state_\control) +
\VALUE\big( t+1,\DYNAMICS(\state,\control) \big) \Big) \)
naturally leads to solutions as \emph{policies} or \emph{strategies},
where an optimal decision $\control$ at time $t$ depends no only on $t$,
but also on  the state $\state(t)$ \citep*{DeLara-Doyen:2008}.

In this section, we shall present a class of strategies called
\emph{index strategies}. Among them, the so-called \emph{Gittins index
strategy} plays a special role, in that it easily provides an upper
bound to the value of the mine. 

\subsection{Index based policy heuristics}

We introduce a technique to
obtain suboptimal results, based on so-called \emph{index strategies}.
The essence of this method is to model the problem by a set of
\emph{jobs}, each job being characterized by its state of progress,
and combined with an index, whose value will indicate the priority of
the job. At each time period, we choose the job of higher index to
work at, which has the effect of modifying its state of progress, and
update its index.

In the open-pit mine scheduling problem, the jobs in question will be
be the vertical columns located by their surface coordinates, and the
state of progress will be the depths of the columns as defined
previously. We define an index which, at each column, will map a value
generally linked with the worth of the blocks around and below the top
block of the column, and that includes or not the slopes constraints.

Various indices can be defined, each one giving a different strategy,
and therefore different results and running times.  Index
algorithms with slope admissibility constraints work as follows. For
each column $\horizontal$ in the block model, and for each local state
\( \state_{\horizontal} \) (attached to the column), a certain
$\INDEX_{\horizontal}(\state_{\horizontal}) \in \RR$ is calculated.
Then, for each column, we check whether or not its top block is
extractable (in terms of the slope constraints). Among the columns
whose top blocks are extractable, we pick the column with highest
index and remove its top block, recalculating the index for that
column. We iterate in this way until all blocks have been extracted,
therefore generating a sequence of blocks.

The index of a column can be any function of the block model. We
consider the following ones (see Figure~\ref{fig:indices} for a few
examples). They correspond to existing heuristics that we interpret in
terms of index. 

\begin{itemize}
  \item The greedy index $\INDEX^g$, that is, the one that uses as
  index the economic value of the top-most block in the column (that
  has not been extracted yet).

  \item The Gittins index $\INDEX^G$, that calculates the maximum
  discounted value of blocks in the column, relative to other columns.
  Block values are discounted block by block.

  \item The best-cone index $\INDEX^{C^*}$. This index is similar to
  the previous one, but calculates all values for the different cones
  truncated at different depths, selecting the one with highest
  value.

  \item Toposort $\INDEX^\tau$. This is the index attached to the
  algorithm proposed by \citep*{CEG+11}. To calculate this index, we
  first solve the linear relaxation of the problem and then set the
  following value for each block
  \[ T_i = T + 1 \ - \ \sum_{t=1}^T t \Delta \byvar{it} + (T+1)
  \left[1-\sum_{t=1}^T \byvar{it} \right].  \]
  Here, $\atvar{it}$ is the binary variable associated to the decision
  of extracting a block $i$ at time period $t$ and
  $\byvar{it}=\sum_{s\le t} \atvar{is}$ (see
  Appendix~\ref{sec:combinatorial} for a detailed formulation). 
  The index then corresponds to the value $T_i$ of the top-most block
  in the column (that has not been extracted yet).  
\end{itemize}

\begin{figure} \begin{center}
\begin{picture}(0,0)%
\includegraphics{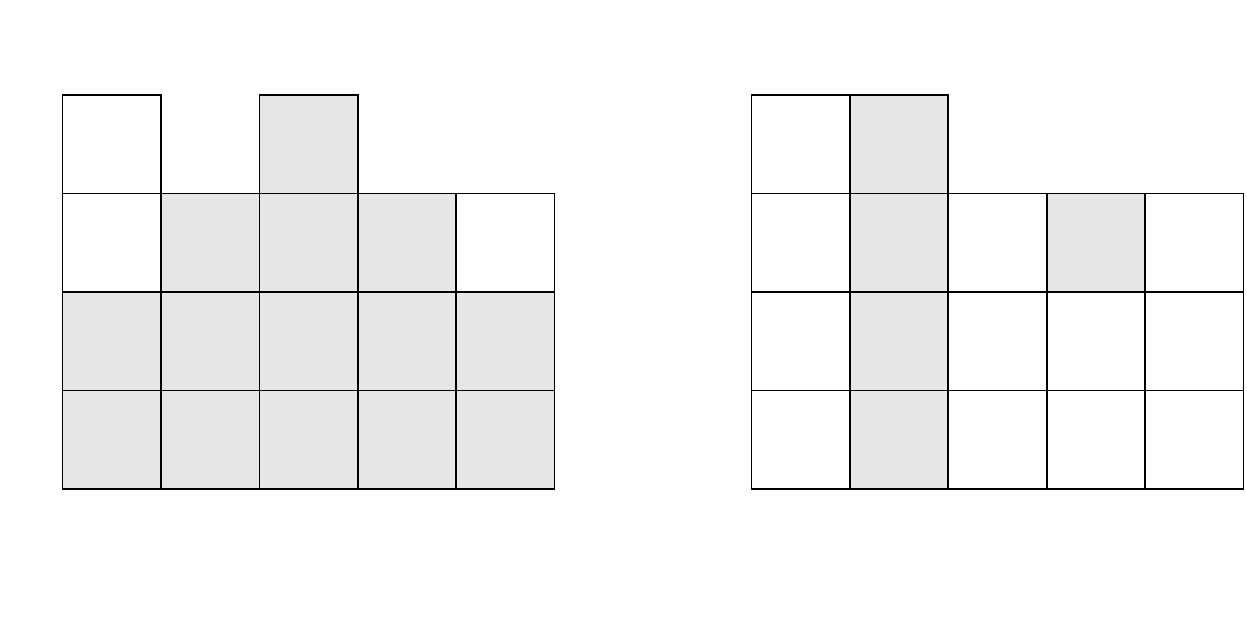}%
\end{picture}%
\setlength{\unitlength}{4144sp}%
\begingroup\makeatletter\ifx\SetFigFont\undefined%
\gdef\SetFigFont#1#2#3#4#5{%
  \reset@font\fontsize{#1}{#2pt}%
  \fontfamily{#3}\fontseries{#4}\fontshape{#5}%
  \selectfont}%
\fi\endgroup%
\begin{picture}(5697,2874)(-284,-1615)
\put(136,929){\makebox(0,0)[lb]{\smash{{\SetFigFont{12}{14.4}{\familydefault}{\mddefault}{\updefault}{\color[rgb]{0,0,0}$c_1$}%
}}}}
\put(1936,929){\makebox(0,0)[lb]{\smash{{\SetFigFont{12}{14.4}{\familydefault}{\mddefault}{\updefault}{\color[rgb]{0,0,0}$c_5$}%
}}}}
\put(586,929){\makebox(0,0)[lb]{\smash{{\SetFigFont{12}{14.4}{\familydefault}{\mddefault}{\updefault}{\color[rgb]{0,0,0}$c_2$}%
}}}}
\put(1036,929){\makebox(0,0)[lb]{\smash{{\SetFigFont{12}{14.4}{\familydefault}{\mddefault}{\updefault}{\color[rgb]{0,0,0}$c_3$}%
}}}}
\put(1486,929){\makebox(0,0)[lb]{\smash{{\SetFigFont{12}{14.4}{\familydefault}{\mddefault}{\updefault}{\color[rgb]{0,0,0}$c_4$}%
}}}}
\put(3286,929){\makebox(0,0)[lb]{\smash{{\SetFigFont{12}{14.4}{\familydefault}{\mddefault}{\updefault}{\color[rgb]{0,0,0}$c_1$}%
}}}}
\put(5086,929){\makebox(0,0)[lb]{\smash{{\SetFigFont{12}{14.4}{\familydefault}{\mddefault}{\updefault}{\color[rgb]{0,0,0}$c_5$}%
}}}}
\put(3736,929){\makebox(0,0)[lb]{\smash{{\SetFigFont{12}{14.4}{\familydefault}{\mddefault}{\updefault}{\color[rgb]{0,0,0}$c_2$}%
}}}}
\put(4186,929){\makebox(0,0)[lb]{\smash{{\SetFigFont{12}{14.4}{\familydefault}{\mddefault}{\updefault}{\color[rgb]{0,0,0}$c_3$}%
}}}}
\put(4636,929){\makebox(0,0)[lb]{\smash{{\SetFigFont{12}{14.4}{\familydefault}{\mddefault}{\updefault}{\color[rgb]{0,0,0}$c_4$}%
}}}}
\put(2881,524){\makebox(0,0)[lb]{\smash{{\SetFigFont{12}{14.4}{\familydefault}{\mddefault}{\updefault}{\color[rgb]{0,0,0}1}%
}}}}
\put(2881, 74){\makebox(0,0)[lb]{\smash{{\SetFigFont{12}{14.4}{\familydefault}{\mddefault}{\updefault}{\color[rgb]{0,0,0}2}%
}}}}
\put(2881,-376){\makebox(0,0)[lb]{\smash{{\SetFigFont{12}{14.4}{\familydefault}{\mddefault}{\updefault}{\color[rgb]{0,0,0}3}%
}}}}
\put(2881,-826){\makebox(0,0)[lb]{\smash{{\SetFigFont{12}{14.4}{\familydefault}{\mddefault}{\updefault}{\color[rgb]{0,0,0}4}%
}}}}
\put(-269,524){\makebox(0,0)[lb]{\smash{{\SetFigFont{12}{14.4}{\familydefault}{\mddefault}{\updefault}{\color[rgb]{0,0,0}1}%
}}}}
\put(-269, 74){\makebox(0,0)[lb]{\smash{{\SetFigFont{12}{14.4}{\familydefault}{\mddefault}{\updefault}{\color[rgb]{0,0,0}2}%
}}}}
\put(-269,-376){\makebox(0,0)[lb]{\smash{{\SetFigFont{12}{14.4}{\familydefault}{\mddefault}{\updefault}{\color[rgb]{0,0,0}3}%
}}}}
\put(-269,-826){\makebox(0,0)[lb]{\smash{{\SetFigFont{12}{14.4}{\familydefault}{\mddefault}{\updefault}{\color[rgb]{0,0,0}4}%
}}}}
\put(991,1109){\makebox(0,0)[lb]{\smash{{\SetFigFont{12}{14.4}{\familydefault}{\mddefault}{\updefault}{\color[rgb]{0,0,0}Mine 1}%
}}}}
\put(4141,1109){\makebox(0,0)[lb]{\smash{{\SetFigFont{12}{14.4}{\familydefault}{\mddefault}{\updefault}{\color[rgb]{0,0,0}Mine 2}%
}}}}
\put( 46,-1231){\makebox(0,0)[lb]{\smash{{\SetFigFont{12}{14.4}{\familydefault}{\mddefault}{\updefault}{\color[rgb]{0,0,0}(a) $\INDEX^{C}{c_3}(1)$}%
}}}}
\put(3241,-1231){\makebox(0,0)[lb]{\smash{{\SetFigFont{12}{14.4}{\familydefault}{\mddefault}{\updefault}{\color[rgb]{0,0,0}(b) $\INDEX^G_{c_2}(1)$}%
}}}}
\put(3241,-1546){\makebox(0,0)[lb]{\smash{{\SetFigFont{12}{14.4}{\familydefault}{\mddefault}{\updefault}{\color[rgb]{0,0,0}(c) $\INDEX^g_{c_4}(2)$}%
}}}}
\end{picture}%
  \caption{Example of index strategies in two small 2-D mines: (a)
  Cone, (b) Gittins and (c) Greedy.}
  \label{fig:indices}
\end{center} \end{figure}

Famous techniques in mining can be interpreted as index strategies.
For example, the Greedy index corresponds to a greedy strategy of
always picking for extraction the block in the surface that: (a) is
extractable (in terms of slope constraints) and (b) has the highest
economic value. Furthermore, the Cone index described before is
close to the {\em Gershon Algorithm} \citep*{Ger1987} which 
also considers the successors' cone, but intersected with the ultimate
pit.

\subsection{An upper bound given by the  Gittins index strategy}

We shall now provide upper and lower bounds to the value~\eqref{eq:intertemporal_profit_maximization}
of the mine by means of index strategies.

To each {profile}
$\state=\big(\state_{\horizontal}\big)_{\horizontal\in\setcolumns} \in
\XX$ and column $\horizontal \in \overline{\setcolumns}$, associate the 
local state \( \state_{\horizontal} \in \{1,\ldots,\DEPTH+1\} \), which 
is the vertical position of the top block with 
horizontal position $\horizontal$.
For \( \discount\upper \in ]0,1[ \), 
define the Gittins index by
\begin{equation}
 \INDEX^G_{\horizontal}(\state_{\horizontal}) \defegal 
\displaystyle \sup_{\tau =t_{0}, \ldots, +\infty}
\frac{ \displaystyle \sum_{s=t_{0}}^{\tau} \discount\upper^{s} 
\rent(\horizontal,\state_{\horizontal} + s ) }%
{\displaystyle \sum_{s=t_{0}}^{\tau} \discount\upper^{s} } \; ,
\end{equation}
where \( \rent(\horizontal, \depth) \defegal 0 \) when 
\( \depth > \DEPTH \) (this corresponds to fictituous blocks with zero values below the mine).
With the notations of \S\ref{ss:state}, 
the Gittins index strategy is defined by
\begin{subequations}
\begin{equation}
\control\Gittins(t) \in \arg \max \{ \INDEX\Gittins_{\horizontal}\big(\state\Gittins_{\horizontal}(t)\big) \, , \quad
\horizontal \in \overline{\CONTROL} \}  \; ,
\end{equation}
\begin{equation}
\state\Gittins(t+1)=\DYNAMICS\big(\state\Gittins(t),\control\Gittins(t)\big) \; .
\end{equation}
\label{eq:Gittins_index_strategy}
\end{subequations}

\begin{proposition}
Suppose that \( \horizon=+\infty \), and 
that the discounting factor function $\discount(t) $ in~\eqref{eq:criterion} satisfies 
\begin{equation}
0 \leq \discount(t) \leq \discount\upper^t < 1 \; . 
\label{eq:discount_inequality}
\end{equation}
The value~\eqref{eq:intertemporal_profit_maximization}
of the mine is bounded above as follows
\begin{equation}
 \Value\opt \leq 
\sum_{t=t_{0}}^{+\infty} \discount\upper^{t} 
\rent\big(\state_{\control\Gittins(t)}(t),\control\Gittins(t)\big) \; ,
\label{eq:upper_bound}
\end{equation}
where the sequence \( \control\Gittins(\cdot) \) is given by the 
Gittins index strategy~\eqref{eq:Gittins_index_strategy} above. 
A lower bound is given by
\begin{equation}
 \Value\big( \control^i(\cdot) \big) \leq  \Value\opt 
\label{eq:lower_bound}
\end{equation}
where the sequence \( \control^i(\cdot) \) is given by  any index strategy respecting slope admissibility constraints 
\begin{subequations}
\begin{equation}
\control^i(t) \in \arg \max \{ \INDEX^i_{\horizontal}\big(\state^i_{\horizontal}(t)\big) \, , \quad
\horizontal \in \dec(\state^i) \}  \; ,
\end{equation}
\begin{equation}
\state^i(t+1)=\DYNAMICS\big(\state^i(t),\control^i(t)\big) \; .
\end{equation}
\end{subequations}
\end{proposition}

\begin{proof}
Recall that \( \Value\opt \) is the maximal value of~\eqref{eq:criterion} among all sequences \( \big( \control(\cdot), \state(\cdot) \big) \) which satisfy the slope constraints~\eqref{eq:control_constraints}. Therefore, \( \Value\opt \) is larger than any  \( \Value\big( \control(\cdot) \big) \), in particular for a sequence \( \control^i(\cdot) \) given by  an index strategy respecting slope admissibility constraints. This is why \eqref{eq:lower_bound} holds true.

On the other hand, by~\eqref{eq:discount_inequality}, we have that 
\[
 \Value\opt \leq 
\max \{ \sum_{t=t_{0}}^{\horizon} \discount\upper^{t}  \rent\big(\state_{\control(t)}(t),\control(t)\big) \, , \quad 
\big( \control(\cdot), \state(\cdot) \big) \, , \, 
\control(t) \in \dec\big(\state(t)\big) \} \; .
\]
Now, if we relax the slope admissibility constraints \( \control(t) \in \dec\big(\state(t)\big) \), we deduce that 
\[
 \Value\opt \leq 
\max \{ \sum_{t=t_{0}}^{\horizon} \discount\upper^{t}  \rent\big(\state_{\control(t)}(t),\control(t)\big) \, , \quad 
\big( \control(\cdot), \state(\cdot) \big)  \} \; .
\]
Gittins theorem \citep*{Gittins:1979} asserts that the optimum for the right hand side is achieved for the 
Gittins index strategy~\eqref{eq:Gittins_index_strategy}.
Indeed, the problem is a deterministic multi-armed bandit, with independent arms since the slope admissibility constraints are relaxed, enabling thus to select any column. This is why \eqref{eq:upper_bound} holds true.
\end{proof}

Let $\NPV_{opt}$ be respectively the optimal $\NPV$,
$\NPV_{ind}$ the $\NPV$ given by any index
strategy respecting the slopes constraints, and $\NPV_{ub}$ the $\NPV$ given by
the Gittins index without slopes constraints, but with a
discounting factor function $\discount(t) $ which
satisfies~\eqref{eq:discount_inequality}. Then we have the following inequality:
\begin{equation}
\NPV_{ind}\le \NPV_{opt}\le \NPV_{ub} \; .
\end{equation}

\section{Numerical examples} 
\label{sec:results}

In this section, we present and discuss numerical results obtained using 
index heuristics over a set of synthetic data and the Marvin block model.

\subsection{The Marvin dataset}

The mine considered for this study is a well known mine named {\em
Marvin}, which is available for use within the mine planning
optimization Whittle from GenCom software.  
The overall number of blocks in Marvin is about 53,000. 
The block model contains the following data: block coordinates
($x,y$ and $z$), copper and gold grades ($\textrm{cooper}_i$ and
$\textrm{gold}_i$ respectively) and density. From these attributes we
calculate: a block tonnage $w_i$ (the product of the density by the
volume of the block) and the copper content (the tonnage of the block
by its copper grade). We aim to maximize overall copper production
under a transportation capacity of 30,000 tons per day.  Finally, we
consider annual time-periods with a yearly discount rate equivalent to
a 10\% opportunity cost, hence a yearly discount factor $\DISCOUNTRATE  = \frac{1}{1 + 0.1}$. 


\subsection{Using block sequences to obtain blocks schedulings}
\label{sec:columns_blocks}

First, we present how to transform the output of an indexing
strategy into a block scheduling and, therefore, a solution of \OPBSP.
We regard the output of an indexing algorithm as a {\em sequence} of
blocks: a block sequence is a tuple of blocks $S=(i_1,
i_2, \ldots, i_N)$ that is compatible with the precedence
constraints. 

A sequence $S$ can be converted into a solution of the open-pit
block sequencing problem with capacity constraints, by creating nested
pits that extract the blocks in the order given by the sequence. More
precisely, let us say that $P \subset \BLOCKSET$ is capacity-feasible
at time period $t$ if for each resource $r$, we have that 
\(
    \sum_{i \in P} a(i,r) \le C^+_{r,t}
\). 
We can then follow the next procedure to construct a block scheduling:

\begin{enumerate}
  \item Set $k=1, t=1, P_0 = P_1 = \emptyset$. 

  \item While $t \le T$:
  \begin{enumerate}
    \item While $k<N$ and $(P_t \cup \{i_k\}) \setminus P_{t-1} $ is
    capacity-feasible at time period $t$: $P_t \leftarrow P_t \cup
    \{i_k\}$, $k \leftarrow k+1$.

    \item $t \leftarrow t+1$.
  \end{enumerate}
\end{enumerate}
Notice, however, that there may exist some room for improvement on the
obtained block scheduling, as it could happen that the blocks assigned
to the very last time-period have a negative overall value. If this is
the case, we reset these blocks as unextracted.

An alternative way to convert a block sequence into a block scheduling
is the following. Given the sequence $S=(i_1, i_2, \ldots, i_K)$, we
set $\BLOCKSET = \{i_k : k=1, 2, \ldots, K\}$ and $\ARCSET = \{ (i_k,
i_{k+1}) : k = 1, 2, \ldots, K-1\}$ and then directly solve the
instance $\OPBSP(\BLOCKSET,\ARCSET,V,A,T,\DISCOUNTRATE,C^+,C^-)$. This is
equivalent to the procedure described above with the last ``cleaning''
phase.

\subsection{Results and discussion}

We now present the different results obtained for the
heuristics and data sets, and we comment the findings of the numerical
experiences.

\begin{table}[h] \begin{center}
  \begin{tabular}{|l|c|c|c|c|c|}
  \hline
  \hline
  Mine & - & Best Index & TopoSort & LP & Index UB \\
  \hline
  \multirow{2}{*}{{\tt PM1}}
  & Time & 0.43s & 307s & 307s & 0.08s \\
  & Value & 358.42 & 432.14 & 439.30 & 521.53 \\
  \hline
  \multirow{2}{*}{{\tt PM2}}
  & Time & 0.27s & 375s & 375s & 0.07s \\
  & Value & 319.74 & 438.60 & 439.84 & 674.24 \\
  \hline
  \multirow{2}{*}{{\tt PM3}}
  & Time & 0.27s & 362s & 362s & 0.07s \\
  & Value & 139.50 & 149.06 & 198.84 & 318.72 \\
  \hline
  \multirow{2}{*}{{\tt Marvin}}
  & Time & 1,036.00s & $\infty$ & $\infty$ & 15.1s \\
  & Value & 392.9 & - & - & 488.5 \\
  \hline
  \hline
  \end{tabular}
  \caption{Numerical results by heuristic and instance. Values in
  million of copper tons. Running time in seconds. LP is linear
  programming. Index UB is index upper bound.}
  \label{tbl:results}
\end{center} \end{table}

Numerical experiences were run with an Intel Pentium Dual Core, 2.8 Ghz
processor running Linux 2.6.30-1. LP's were solved using the GNU
Linear Programming Toolkit ({\tt GLPK}) using the primal simplex
method.

Results in running time and economic value ($\NPV$) are presented in
Table~\ref{tbl:results}.
We observe that, while TopoSort obtains better results (closer to
the LP upper bounds), this approach does not ``scale'' well, as it does not
produce feasible solutions for the Marvin instances. Indeed, the main
difficulty in this case is to solve the Linear Relaxation (LP), which
did not end within reasonable time (12 hours). 
Conversely, the index strategies provide mixed results for the bounds, but the
execution time is quite small, making them good candidate for fast
schedulers and therefore useable with uncertainty scenarios, for
example, on the grades.

We observe that there is a lot of room to improve the speed of the
heuristics by optimizing the code or, for example, parallelizing some
of the computations.

\section{A mathematical framework for mine scheduling under uncertainty}
\label{sec:uncertainty-framework}

We present here a general probabilistic framework for the \OPBSP,
that allows a dynamical use of information (learning), permitting to develop
adaptive strategies, and which includes the planning solutions as a
particular case.  The approach is mostly mathematical
and formal. However, in the last part, we suggest possible heuristics
for future research.

\subsection{Block attributes}

Denote discrete time by $t=t_{0},\ldots,\horizon$, where the horizon 
$\horizon$ is supposed to be finite for simplicity.
Denote by $\setblocks$ the set of all blocks. 
Each block $\block\in\setblocks$, when extracted in period $t$, is
characterized by a $\nba$-vector of \emph{attributes}
$\uncertainty_{\block}(t)\in \Uncertain=\RR^\nba$. These attributes
can for instance be the rock and ore volumes, price, cost, etc. In the
deterministic model, these values will be simple real numbers
perfectly known, but in our case it will be an uncertain vector. 

This uncertain vector $\uncertainty_{\block}(t)$ will summarize
various sources of uncertainty, and will be the basis of the
construction of the worth $\worth_{\block}(t)$ of block $\block$ at time
$t$.
It can
for instance be of the following form, if the mine contains $d$
different ores, \begin{equation}
\uncertainty_b(t)=(\price(t),\amount(\block),\cost(\block,t),\ldots)
\label{eq:modelattributes} \end{equation} \begin{equation}
\worth_b(t)=\price(t)\cdot\amount(\block)-\cost(\block,t) \; , 
\label{eq:modelvalue} \end{equation}
where $\price(t)\in\RR^d$ is an uncertain vector representing the
selling prices per unit of the $d$ different ores at time $t$,
$\amount(\block)\in\RR^d$ is an uncertain vector representing the
amount of each ore in the block $\block$, and $\cost(\block,t)$ is a
uncertain variable representing the extraction cost of the block
$\block$ at time $t$, each of them being coordinates of the attributes
vector $\uncertainty_{\block}(t)$. This formulation presents the
advantage to split the price ditribution modeling and the distribution
of the different ores in the mine; it is of course a simple instance
that can be replaced by more sophisticated models including processing
costs or other geotechnical data.

\subsection{Scenarios}

In the sequel, we will use the following notations
\begin{equation*}
\uncertainty(t):=(\uncertainty_{\block}(t))_{\block\in\setblocks}
\end{equation*}
for the collection of the attributes of the mine blocks at a time period $t$.
while
A sequence
\begin{equation*}
\uncertainty(\cdot):=(\uncertainty(t_0),\hdots,\uncertainty(\horizon))
\end{equation*}
is called a \emph{scenario} and belongs to the product set
\begin{equation*}
\Omega:=\prod_{t=t_0}^{\horizon}\prod_{\block\in\setblocks}
\Uncertain = \RR^{\numb\cdot(\horizon-t_0+1)\cdot\nba} \; , 
\end{equation*}
which is the set of all possible scenarios. The situation where $\Omega$ is a singleton (a unique scenario) corresponds to the deterministic case.

\subsection{A priori information data on the scenarios}

Additional a priori information on the scenarios is generally given 
either by probabilistic or by set membership settings.

\subsubsection*{Stochastic assumptions}

Notice that the vectors $\uncertainty_{\block}(t)$ are {a priori} not
independent, neither with respect to $\block$ (spatially), nor with
respect to $t$ (temporally). Indeed, the price of raw materials is
highly correlated in time, and a strong spatial correlation exists in
the repartition of the ore. Many models of the orebody are based on the
notion of variogram, which is a geostatistical tool giving an index of
the spatial correlation of a certain type of ore. It gives a
represention of the typology of the ore in a site, some metals as gold
tending to aggregate into nuggets (with a strong short-distance
correlation but a lower long-distance one), whereas other like copper
will have a more long-distance dependence. It opens the way to orebody
modeling such as kriging, a widespread interpolation method in
geostatistics.  

In the probabilistic formalism, the set $\Omega$ of all scenarios is equipped with the Borel $\sigma$-field \( \tribu{F}=\borel{\RR^{\numb\cdot(\horizon-t_0+1)\cdot\nba}}\).
The $\uncertainty_\block(t)$ becomes random vectors, and the orebody is represented by a joint distribution law 
\begin{equation}
\mathcal{L}(\uncertainty_\block(t),\block\in\setblocks,t\in[t_0,...,\horizon])
\; ,
\label{eq:loijointe}
\end{equation}
which is a probability on \( ( \Omega, \tribu{F} ) \).
For instance, in the case of a unique type of ore, we can model $(\amount(\block))_{\block\in\setblocks}$ by a Gaussian vector of size $\numb$, characterized by its mean vector $\mu=(\EE[\amount(\block)])_{\block\in\setblocks}$ and its covariance matrix $\Sigma=(Cov(\amount(\block),\amount(\block')))_{\block,\block'\in\setblocks}$, with constant price $\price(t)=\price$ and cost $\cost(\block,t)=\cost$. The set of the worths $\worth_{\block}(t),\block\in\setblocks, t\in[t_0,...,\horizon]$, is then a Gaussian vector of size $\numb\cdot(\horizon-t_0+1)$ whose mean vector and covariance matrix can be calculated by means of $\mu$ and $\Sigma$.

\subsubsection*{Set membership}

 For a given block $\block$ and a given time period $t$, $\uncertainty_{\block}(t)$ can take its value in a certain set $\mathbb{S}(\block,t)\subset\RR^\nba$, which depends on the model. In the most general case, if we know nothing about the mine, $\mathbb{S}(\block,t)$ will be $\RR^\nba$, but it can for instance be reduced to intervals or even to a finite number of values, or to a singleton in a deterministic model.

\subsection{Decisions and constraints}

Each period of time (year, for instance), we can extract a certain number of blocks, and therefore we model our decision by a variable $\deci(t)\in\setdecis = 2^\setblocks$, corresponding to the blocks removed at time $t\in[t_0,...,\horizon]$, which form a subset of $\setblocks$. 
Here, $2^\setblocks$ denotes the set of subsets of $\setblocks$ (the power set  of $\setblocks$).
Since $\setdecis$ is a finite set, we equip it with the complete 
$\sigma$-field \( \tribu{U}=2^\setdecis \). 
We introduce the notations:
\begin{equation*}
\deci^t \defegal \big( \deci(t_0),\hdots,\deci(t) \big)
\text{and}
\deci(\cdot) \defegal \big(\deci(t_0),\hdots,\deci(\horizon)\big) \; .
\end{equation*}
The set $\HISTORY \defegal \Omega\times\setdecis^{\horizon-t_0+1}$ is called the \emph{history space}. 
Elements of the set $\HISTORY_t \defegal \Omega\times\setdecis^{t-t_0+1}$ represent history up to time $t$.

To capture {slope and uncertain capacity constraints}, we can restrict decisions as belonging to a subset 
\( \setdecis\big(t,\uncertainty(\cdot),\deci^{t-1}\big) \) of \( \setdecis \)
as follows:
\begin{equation}
\deci(t) \in \setdecis\big(t,\uncertainty(\cdot),\deci^{t-1} \big) \, .
\label{eq:constraints}
\end{equation}

\subsection{On-line information}

After having seen a priori information data on the scenarios, we now turn to 
on-line information available for the planner at time $t$. In essence, it is built upon the attributes $(\uncertainty_b(t))_{\block,t}$ we have discovered, and thus it {a priori} also depends on the past extractions $\deci^{t-1}$ (i.e. the choices done on $[t_0,...,t-1]$). Mathematically, we shall represent information at time $t$ as a $\sigma$-algebra $\info_t$ on the {history space} $\Omega\times\setdecis^{\horizon-t_0+1}$.

\begin{itemize}
\item The {blind} information pattern is 
\begin{equation*}
\info_t=\{\UNCERTAIN,\varnothing\}\otimes
\{\setdecis^{\horizon-t_0 +1},\varnothing\} \; ,
\end{equation*}
where the decision-maker cannot distinguish elements in the history space (he cannot even recall his past decisions).
\item 
The anticipative point of view corresponds to a stationary and constant
\begin{equation}
\info_t = \tribu{F} 
\otimes \{\setdecis,\varnothing\} \; .
\end{equation}
The decision-maker knows the attributes of each block at each time, and knows them in advance: he is a visionary decision-maker. 
A visionary decision-maker having recall of his past decisions would be modeled as \( \info_t = \tribu{F}
 \otimes \bigotimes_{s=t_0}^{t-1} \tribu{U} \).
\item 
A \emph{causal} information pattern is one in which the decision-maker
cannot base his decision at time $t$ upon his future decisions, and it is represented by the condition
\begin{equation}
\info_t \subset \tribu{F}
 \otimes \bigotimes_{s=t_0}^{t-1} \tribu{U}  \; .
\label{eq:causal}
\end{equation}
\item In the \emph{cumulative information pattern}, 
let us denote by
\begin{equation*}
\State(t,u^{t-1}) \defegal \cup_{s=t_0}^t\deci(t)\subset\setblocks
\end{equation*}
the set of the blocks which have been removed at time $t$ following the
sequence $u^{t-1}$ of decisions in the periods $[t_0,...,t-1]$.
If we assume that, each time we extract a block $\block$ at period $t$,
we learn the exact value of the uncertainty $\uncertainty_b(t)$,
we define the information as
\begin{equation}
\info_t =\sigma \{ (\uncertainty_{\block}(s),\deci^{s-1}) \, , \quad \block\in \State(\deci^{s-1},s), s\in[t_0,...,t-1] \} \; ,
\label{eq:infoexemple}
\end{equation}
where we have abusively identified \( (\uncertainty_{\block}(s),\deci^{s-1}) \) 
with the coordinate random variable on the {history space} $\Omega\times\setdecis^{\horizon-t_0+1}$.

This formulation is adapted to a dynamical strategy, in which we learn step-by-step the information depending on our past choices.
\end{itemize}

\subsection{Adaptive strategies}

We now have the tools to define strategies adapted to on-line information.
We assume that the information pattern is causal, that is, satisfies~\eqref{eq:causal}.
A (causal) \emph{strategy} is a sequence \( \policy = \big( \policy_t \big)_{t = t_0, \ldots, \horizon} \) of \emph{policies}
\begin{equation*}
\policy_t : \Omega\times\setdecis^{t-t_0}  \rightarrow\setdecis
\end{equation*}
such that, for all $t = t_0, \ldots, \horizon$, $\policy_t $ is measurable with respect to $\info_t$.

Once a strategy \( \policy \) and a scenario \( \uncertainty(\cdot) \) are given, decisions are inductively deduced by
\begin{equation}
 \deci(t)=\policy_t(\uncertainty(\cdot), \deci^{t-1}) \; .
\end{equation}
Now, strategies will be our optimization variables.

If the family of sets \( \setdecis\big(t,\uncertainty(\cdot),\deci^{t-1} \big) \) in~\eqref{eq:constraints} is 
 measurable with respect to $\info_t$, we may restrict ourselves to 
strategies in the admissible set $\setpolicies$ of the policies compatible with the constraints (capacity constraints, slopes constraints, etc.).
For instance, a capacity constraints of $k$ blocks per time unit will
imply that, for $\policy\in\setpolicies$, the $\deci(t)$ generated by
$\policy$ will not be more than $k$, or for a certain type of slopes constraints and precedence extraction relations, that the decisions $\deci(t)$ generated by $\policy$ will be compatible with the constraints induced by the blocks $\State(t,\deci^{t-1})$ already removed.

A strategy $\policy\in\setpolicies$ is said to be an \emph{open-loop} strategy if $\policy_t$ is a constant mapping for all $t$. In other words, an open-loop strategy plans the entire extraction sequence before starting it, and does not modify the sequence even if one gets information over time. In the more general case in which $\policy$ depends on the information, the strategy is said to be a \emph{closed-loop} strategy. It corresponds to the adaptive case.



\subsection{Decision criteria under uncertainty}

For a given scenario \( \uncertainty(\cdot) \) and 
a given control sequence \( \deci(\cdot) \), 
the sum of discounted profits (NPV) 
is given by 
\begin{equation}
\gain\big(\uncertainty(\cdot),\deci(\cdot)\big) =
\sum_{t=t_0}^{\horizon}\discount(t)\sum_{\block\in\deci(t)}\worth_{\block}(t)\; .
\label{eq:uncertain_criterion}
\end{equation}
For a given scenario \( \uncertainty(\cdot) \) and 
a given strategy \( \policy \) (adapted to the information pattern $\info_t$, $t=t_0, \ldots, \horizon$), let us put
\begin{equation}
 \gain^\policy\big( \uncertainty(\cdot) \big) \defegal 
\gain\big(\uncertainty(\cdot),\deci(\cdot)\big) \mtext{ where }
 \deci(t)=\policy_t(\uncertainty(\cdot), \deci^{t-1}) \; .
\label{eq:policy_uncertain_criterion}
\end{equation}
Now, contrarily to deterministic optimization, we do not know in advance the 
scenario \( \uncertainty(\cdot) \). How the decision-maker aggregates~\eqref{eq:policy_uncertain_criterion} with respect to the uncertainties, before optimizing, 
reflects his sensitivity to risk. The most common aggregates are the \emph{robust} (or worst-case) and the \emph{expected} criteria, but we also present other examples.

\begin{itemize}
\item
\subsubsection*{The expected criterion}

The expected optimization problem is
\begin{equation}
\max_{\policy\in\setpolicies}\EE^{\PP}[\gain^\policy\big( \uncertainty(\cdot) \big) ]
\; ,
\label{eq:expcNPV}
\end{equation}
where $\EE^{\PP}$ denotes the mathematical expectation with respect to a probability $\PP=\mathcal{L}(\uncertainty_\block(t),\block\in\setblocks,t\in[t_0,...,\horizon])$ on the space $\Omega$ of scenarios.
This formulation aims to maximize the mean NPV, that is the average value of all possibilities, weighted by their probability to happen. It is the best formulation you can choose in terms of average gain, but it does not penalize the possible realizations of the worst cases.

\item
\subsubsection*{The robust criterion}

The robust optimization problem is
\begin{equation}
\max_{\policy\in\setpolicies} \min_{\uncertainty(\cdot)\in\Omega} \gain^\policy\big( \uncertainty(\cdot) \big) \ .
\label{eq:worst}
\end{equation}
The strategy given by this formulation ensures to maximize the NPV if the worst case happens.

\item
\subsubsection*{The multi-prior approach}

Suppose that the space $\Omega$ of scenarios is equipped with different probabilities $\PP$ in a set $\mathfrak{P}$, reflecting \emph{ambiguity} with respect to the stochastic model.  The multiprior approach is a combination of the robust and the expected criteria by taking the worst belief in term of expected NPV:
\begin{equation}
\max_{\policy\in\setpolicies}\min_{\PP\in\mathfrak{P}}\EE^{\PP}[\gain^\policy\big( \uncertainty(\cdot) \big) ]\ .
\label{eq:multiprior}
\end{equation}

\item
\subsubsection*{An expected criteria under probability constraint}

This last formulation is similar to the maximization of the expected
NPV, but with an additional constraint to handle the risk.
Given two parameters $\alpha\in \RR$ and $p\in [0,1]$, the expected optimization problem under probability constraint is
\begin{equation}
\max_{\policy\in\setpolicies}\EE^{\PP}[\gain^\policy\big( \uncertainty(\cdot) \big) ]
\label{eq:additional}
\end{equation}
under the restriction that 
\begin{equation}
\PP[\gain^\policy\big( \uncertainty(\cdot) \big) \le \alpha]\le p \; .
\end{equation}
The meaning of this formulation is to maximize the expected profit, with
the condition that the chosen strategy will give, with high probability
$1-p$, at least a certain gain~$\alpha$. 
\end{itemize}

Risk measures (Value-at-Risk, Conditional Value-at-Risk, etc.) could also be taken for aggregation \citep*{Follmer-Schied:2002}. 

\subsection{From planning towards adaptive solutions}

As we have seen, since the nineties, a certain number of
``scenario-based strategies'' have been proposed in the literature. The
common denominator of these approaches is the use of conditional
simulation (or any other simulation method), using the distribution law
of the orebody, to generate a set of representative scenarios of the
mine. Then, the solution is generally searched as a planning, that is an open-loop strategy.

A schematic way to represent the elaboration of a scenario-based strategy is the following
\begin{equation}
\mathcal{L} \xrightarrow{sample} (\uncertainty_j(\cdot))_{j\in\mathcal{J}} \xrightarrow{compute} \deci(\cdot)\ ,
\label{eq:schemascenario}
\end{equation}
that is, we sample the distribution law to obtain a set $\mathcal{J}$ of scenarios. Then, with one or another method, we use these scenarios to elaborate an open-loop decision sequence $\deci(\cdot)$.

We suggest that this approach may be extended in the spirit of the \emph{open-loop with feedback control} (OLFC) \citep*{Bertsekas:2000}. 
We do not detail the mathematics, but simply sketch the method.
In the probabilistic setting, we assume that the arrival of an observation at time $t$ allows us to update the conditional distribution $\mathcal{L}^t$ on 
the space $\Omega$ of scenarios, knowing past observations. Then, the
sketch is 
\begin{subequations}
\begin{equation*}
\mathcal{L}^0 \xrightarrow{sample} (\uncertainty_j(\cdot))_{j\in\mathcal{J}^1} \xrightarrow{compute} (\deci^1(1),\ldots,\deci^1(\horizon)) \xrightarrow{select} \deci^1(1)
\end{equation*}
\begin{equation*}
\hookrightarrow \mathcal{L}^1 \xrightarrow{sample} (\uncertainty_j(\cdot))_{j\in\mathcal{J}^2} \xrightarrow{compute} (\deci^1(1),\deci^2(2),\ldots,\deci^2(\horizon))\xrightarrow{select} \deci^2(2)
\end{equation*}
\begin{equation*}
\cdots
\end{equation*}
\begin{equation*}
\hookrightarrow \mathcal{L}^{\horizon-1} \xrightarrow{sample} (\uncertainty_j(\cdot))_{j\in\mathcal{J}^{\horizon}} \xrightarrow{compute} (\deci^1(1),\ldots,\deci^\horizon(\horizon)) \xrightarrow{select} \deci^{\horizon}(\horizon)
\end{equation*}
\label{eq:schemaolfc}
\end{subequations}
returning a closed-loop strategy $\deci(\cdot)$.

To end this section, let us stress the fact that index methods are well adapted to the uncertain case, where the index may be a function of 
the conditional distribution $\mathcal{L}^t$.

\section{Conclusions} \label{sec:conclusions}

We have presented the dynamic optimization approach to the open-pit
block scheduling problem, a relevant
problem in the mining industry that remains ellusive to be solved due
to its size. 
We have proposed heuristics based on so-called index strategies,
together with upper and lower bounds for the NPV.
Some of the results are promising, and index strategies are very fast
and scale well for large instances of mines. This encourages their use when one
generates a large number of scenarios, for which case a fast planning
simulation and NPV calculation is crucial.
In the future, we expect to do more experimentation on larger case studies
and other (more realistic) data sets, and to compare the results with
others found in the literature.

We have also introduced a general framework to deal with uncertainty and
dynamical learning. We expect to implement
this framework and to test it against real data.

\paragraph*{Acknowledgments.}

The authors thank the STIC-AmSud OVIMINE project for the financial support.
This paper was exposed at several OVIMINE meetings ---
2011, March 17-18, Lima, Per\'u,
2011, September 2-19, Paris, France,
2011, October 6-8, Lima, Per\'u,
2012, January 9, Valparaiso, Chile,
2012, November 8, Santiago, Chile ---
and we thank the participants for their comments.

\newcommand{\noopsort}[1]{} \ifx\undefined\allcaps\def\allcaps#1{#1}\fi

\appendix

\section{Integer linear programming formulation of the open-pit block scheduling problem}
\label{sec:combinatorial}

In this section we introduce the relevant notation and formulation
for the deterministic case of the open-pit block scheduling problem
using binary linear programming. 

\subsection{Modeling and notation}

We consider $\BLOCKSET$ the set of all blocks and $N=|\BLOCKSET|$. We
denote the elements of $\BLOCKSET$ (the blocks) with indices $i, j$,
unless otherwise stated. Similarly, we consider $T \in \IN$
time-periods and denote individual time-periods with $s, t = 1, 2,
\ldots, T$. $T$ is called the {\em time horizon}. We also use the
notation ${\bf T} = \{1, 2, \ldots, T\}$ for the set of time-periods.

Slope constraints are modeled as precedence constraints and encoded as
a set of {\em arcs} $\ARCSET \subset B \times B$, so $(i,j) \in
\ARCSET$ means that Block $j$ has to be extracted before Block $i$. We
say, in this case, that Block $j$ is a predecessor of Block $i$, which
in turn is a successor of $j$. Notice that arc $(i,j)$ goes from the
successor to the predecessor.

In this work we address a simplified version of the problem in which
the decision of the destination of the block is done beforehand. This
allows us to

\begin{enumerate}
\item 
consider that the net profit (which can be negative) of processing
Block $i$ is already known and noted as $v_i \in \IR$, and

\item 
define a set of resources $\RESSET$, and for Block $i \in 
\BLOCKSET$ and Resource $r \in \RESSET$ the quantity $a(i,r)$ of
resource $r$ that is used when $i$ is processed.
\end{enumerate}

For each time period $t$, upper and lower bounds on the consumption of
resource $r$ are given by the quantities $C^-_{rt} \in \{-\infty\}
\cup \IR$ and $C^+_{rt} \in \{+\infty\} \cup \IR$, respectively.

We also assume that the block is processed in the same time period in
which it is extracted from the mine (that is, we do not allow to stock
material for future processing). We also assume, as is usual in these
models, that all block extraction, handling and processing is done
within a time-period length.

While the modeling can be easily extended to the general
case, the heuristics presented in this article do not always work to
the case in which blending constraints apply, therefore, we assume
there are not such constraints.

Finally, Table~\ref{tbl:notation} summarizes the notation introduced
in this Appendix.

\begin{table} \begin{center}
  \begin{tabular}{|c|l|}
  \hline
  Symbol & Meaning \\
  \hline
  $\BLOCKSET$ & The set of blocks \\
  $i, j$ & Blocks (elements of $\BLOCKSET$) \\
  $s, t$ & Time-periods \\
  $T$ & Time horizon (number of periods) \\
  $\TSET$ & Set of time-periods \\
  $\ARCSET$ & Set of precedence arcs \\
  $\RESSET$ & Set of resources \\
  $v_i$ & Economic value (net profit) of Block $i$ \\
  $a(i,r)$ & Consumption of Resource $r$ by Block $i$ \\
  $C^-_{rt}, C^+_{rt}$ & Lower and upper bounds on resource $r$\\
  \hline
  \end{tabular}
  \caption{Main notations in the Appendix} \label{tbl:notation}
\end{center} \end{table}


A {\em block scheduling} is a function $\tau : \BLOCKSET \rightarrow \{1, 2,
\ldots, T, \infty \}$ where $\tau(i)$ is the time-period in which
block $i$ is extracted, hence, a block scheduling must satisfy the
precedence constraints, that is if $(i,j) \in \ARCSET$ then $\tau(i)
\ge \tau(j)$.

If $\tau$ is a block scheduling then the preimage sets $P_1 =
\tau^{-1}(1)$ and $P_t= P_{t-1} \cup \tau^{-1}(1)$ for $t>1$ are
called {\em pits}. We observe that $P_t \subset P_{t+1}$ hence we
say that the pits are {\em nested}.

A {\em block sequence} is a tuple $s=(s_1, s_2, \ldots, s_K) \in \BLOCKSET^K$ 
such that $k \neq \ell \Rightarrow s_k \neq s_{\ell}$ (all blocks in
the tuple are different) and that is compatible with the precedence
constraints, that is if $(s_k,s_\ell)=(i,j) \in \ARCSET$ then $\ell > k$ 
(predecessors appear before in the sequence).

\subsection{The binary programming formulation} \label{sec:problem}

The open-pit block scheduling problem is defined on the following
variables. For each $i \in \BLOCKSET, t=1, 2, \ldots, T$:

\[
  \byvar{it} = \left\{ \begin{array}{cl}
    1 & \textrm{block $i$ is extracted by time-period $t$}, \\
    0 & \textrm{otherwise.}
  \end{array}\right.
\]
Notice that the interpretation of variable $\byvar{it}$ is {\em by}
time-period, that is $\byvar{it} = 1$ if and only if block $i$ has been
extracted (and processed) at some period $s$ with $1 \le s \le t$. For
this reason, it is also useful to introduce the following auxiliary
variables for any $i\in \BLOCKSET$: $\atvar{i1} = \byvar{i1}$, and $\atvar{it} = \byvar{it} -
\byvar{i,t-1}$ for $t=2,3,\ldots,T$. We have that $\byvar{it} = \sum_{s=1}^t
\atvar{is}$ and $\atvar{it} = 1 $ if and only if block $i$ is
extracted exactly at time period $t$.

The optimization program is the following:
\begin{alignat}{4}
  (\OPBSP) \max \quad & 
    \sum_{t=1}^T \DISCOUNTRATE^t \sum_{i=1}^N v_i \atvar{it}
      \label{p:target} \\
  & \byvar{it} \le \byvar{jt} & 
    \quad (\forall (i,j) \in \ARCSET)(\forall t \in \TSET) 
      \label{p:precedences} \\
  & \byvar{i,t-1} \le \byvar{it} & \quad
    (\forall i \in \BLOCKSET)(\forall t=2,\ldots,T) 
      \label{p:mass} \\
  & \sum_{i} a(i,r) \atvar{it} \le C^+_{rt} &
    (\forall r \in \RESSET) (\forall t \in \TSET) 
      \label{p:resources+} \\
  & \sum_{i} a(i,r) \atvar{it} \ge C^-_{rt} &
    (\forall r \in \RESSET) (\forall t \in \TSET) 
      \label{p:resources-} \\
  & \byvar{it} \in \{0,1\} & 
    (\forall i \in \BLOCKSET) (\forall t \in \TSET) \; . \nonumber
\end{alignat}
Equation \eqref{p:target} presents the goal function, which is the
discounted value of extracted blocks over the time horizon $T$.
Equation \eqref{p:precedences} corresponds to the precedence
constraints given by the slope angle. Equation \eqref{p:mass} states
that blocks can be extracted only once. Finally, Equations
\eqref{p:resources+} and \eqref{p:resources-} fix the resource
consumption limits.

For a block model $\BLOCKSET$, precedence arcs $\ARCSET$, block
values $V = (v_i)_{i \in \BLOCKSET}$ and attribute matrix
$A=(a(i,r))_{i\in\BLOCKSET,r\in\RESSET}$ we will use the notation
$\OPBSP(\BLOCKSET,\ARCSET,v,A,T,\DISCOUNTRATE,C^+,C^-)$ to denote an
instance of the open-pit block scheduling problem for a certain time
horizon $T$, discount rate $\DISCOUNTRATE$, and resource limit
matrices $C^- = (C^-_{r,t})_{r,t}$ and $C^+ = (C^+_{r,t})_{r,t}$. We
will omit some of the parameters if the context allows it.
\end{document}